\documentclass[12pt]{extarticle}
\pdfoutput=1
\usepackage{amsmath, amsthm, amssymb, mathtools, hyperref, color}
\usepackage{mathrsfs}
\usepackage{epigraph}
\usepackage[shortlabels]{enumitem}
\usepackage{graphicx}
\usepackage{adjustbox}
\usepackage{xspace}
\usepackage{tikz}
\usetikzlibrary{arrows,calc,fit,matrix,positioning,cd}
\usepackage{tikz-3dplot}

\definecolor{darkblue}{RGB}{0,0,160}
\hypersetup{
colorlinks,%
citecolor=black,%
filecolor=black,%
linkcolor=darkblue,%
urlcolor=darkblue,
pdfauthor={Thomas Kahle, Lukas Kühne, Leonie Mühlherr, Bernd Sturmfels, Maximilian Wiesmann},
pdftitle={Arrangements and Likelihood}
}

\oddsidemargin \evensidemargin
\newtheorem{theorem}{Theorem}
\numberwithin{theorem}{section}
\newtheorem{proposition}[theorem]{Proposition}
\newtheorem{lemma}[theorem]{Lemma}
\newtheorem{corollary}[theorem]{Corollary}
\newtheorem{conjecture}[theorem]{Conjecture}

\newtheorem{problem}[theorem]{Problem}

\theoremstyle{definition}
\newtheorem{definition}[theorem]{Definition}
\newtheorem{remark}[theorem]{Remark}
\newtheorem{example}[theorem]{Example}

\newcommand{\RR}{\mathbb{R}}
\newcommand{\QQ}{\mathbb{Q}}
\newcommand{\ZZ}{\mathbb{Z}}
\newcommand{\NN}{\mathbb{N}}
\newcommand{\PP}{\mathbb{P}}
\newcommand{\CC}{\mathbb{C}}

\newcommand{\cA}{\mathcal{A}}

\newcommand{\cR}{\mathcal{R}}
\newcommand{\cL}{\mathcal{L}}

\newcommand{\ot}{\leftarrow}

\DeclareMathOperator{\im}{im}
\DeclareMathOperator{\codim}{codim}
\DeclareMathOperator{\coker}{coker}
\DeclareMathOperator{\Sym}{Sym}
\DeclareMathOperator{\reg}{reg}

\DeclareMathOperator{\pd}{pd}

\DeclareMathOperator{\depth}{depth}
\DeclareMathOperator{\rk}{rk}
\DeclareMathOperator{\mldeg}{MLdeg}
\DeclareMathOperator{\Hom}{Hom}

\renewcommand{\>}{\right\rangle}
\newcommand{\<}{\left\langle}

\newcommand\set[1]{\left\{\,#1\,\right\}}  % Thin spaces are recommended by D. Knuth.
\newcommand\ideal[1]{\left\langle\,#1\,\right\rangle}  % Ideals
\newcommand\iso{\cong}
\newcommand\Mtwo{\textsc{Macaulay2}\xspace}

\newcommand\scalemath[2]{\scalebox{#1}{\mbox{\ensuremath{\displaystyle #2}}}}

\date{}
\allowdisplaybreaks[1]

\newcommand{\logDer}{D_{\operatorname{log}}} % Log derivations module
\newcommand{\Der}{D} % The original Derivations module
\newcommand{\JSyz}{S} % Jacobian Syzygy module

\title{\textbf{Arrangements and Likelihood}}

\author{Thomas Kahle, Lukas Kühne, Leonie
Mühlherr, \\ Bernd Sturmfels and Maximilian Wiesmann \bigskip \\
\small\emph{Dedicated to the memory of Andreas Dress}}

\begin{document}
\maketitle

\begin{abstract}
\noindent
We develop novel tools for computing
the likelihood correspondence of an arrangement of
hypersurfaces in a projective space.
This uses the module of  logarithmic derivations.
This object is well-studied in the linear case, when the hypersurfaces are
hyperplanes. We  here focus on  nonlinear scenarios
and their applications in statistics and physics.
\end{abstract}

\setcounter{tocdepth}{1}

\section{Introduction}
\label{sec:intro}
This article establishes connections between arrangements of
hypersurfaces \cite{Dup, orlik2013arrangements} and likelihood
geometry \cite{HuhSt}.  Thereby arises a new description, summarized
in Theorem~\ref{thm:evaluation}, of the prime ideal $I(\cA)$ of the
likelihood correspondence of a parametrized statistical model. The
description rests on the Rees algebra of the \emph{likelihood module}
$M(\cA)$ of the arrangement~$\cA$, a module that is closely related to
the module of logarithmic derivations introduced by Saito~\cite{Saito}
for a general hypersurface.  Terao's pioneering work~\cite{Terao} for
hyperplane arrangements is by now the foundation of their algebraic
study.  We prove the following result.
\begin{theorem}\label{thm:twoIdeals}
The quotient $R[s]/I(\mathcal{A})$ is the Rees algebra of the
likelihood module $M(\cA)$.
\end{theorem}

In Section~\ref{sec:arrangements}, we start by reviewing Rees algebras
for modules \cite{EHU, SUV} and then prove the theorem.  The nicest
scenario arises when the Rees algebra agrees with the symmetric
algebra.  We call an arrangement $\cA$ \emph{gentle} if the likelihood
module $M(\cA)$ has this property.  In this case, the ideal of the
likelihood correspondence is easy to compute, and the maximum
likelihood (ML) degree is determined by $M(\cA)$.  Being gentle is a
new concept that is neither implied nor implies known properties of a nonlinear
arrangement $\cA$, like being {\em free} or {\em tame}.

The literature on the ML degree \cite{CHKS, HKS} has
focused mostly on implicitly defined models.
We here emphasize the parametric description that is 
more common in statistics, and also seen for scattering equations in
   physics \cite{Lam, ST}.  We 
develop these
connections in Section~\ref{sec:likelihood}.

In Section~\ref{sec:3adjectives} we relate gentleness to the
familiar notions of free and tame arrangements.
Theorem \ref{thm:tlaag} offers a concise statement.
In  Section~\ref{sec:hyperplanes} we turn to the linear case
when the hypersurfaces are hyperplanes.
We study the likelihood correspondence
for graphic arrangements, that is, sub-arrangements of the braid arrangement.
The edge graph of the octahedron yields the
 smallest graphic arrangement which is not gentle; see
Theorem \ref{thm:octahedron}.
 In Section~\ref{sec:software} we present
 software in \Mtwo~\cite{M2} for computing the likelihood correspondence   of $\cA$.

\section{Arrangements and modules}
\label{sec:arrangements}
An \emph{arrangement of hypersurfaces} $\cA$ in projective space
$\PP^{n-1}$ is given by homogeneous polynomials $f_1,f_2,\dotsc, f_m$ in
$R= \CC[x_1,\dotsc, x_n]$.  We work over the complex numbers $\CC$,
with the understanding that the polynomials $f_i$ often have their
coefficients in the rational numbers~$\QQ$. 

 For any complex vector
$s = (s_1,s_2,\dotsc, s_m) \in \CC^m$, we consider the \emph{likelihood
function}
\[
f^s \,=\, f_1^{s_1} f_2^{s_2} \cdots f_m^{s_m}.
% \prod_{i=1}^n f_i^{s_i}.
\]
This is known as the \emph{master function} in the literature on
arrangements \cite{CDFV}.  Its logarithm
\[ \ell_\cA \,=\, 
s_1 \log(f_1) + 
s_2 \log(f_2) + \cdots + 
s_m \log(f_m) \]
% \sum_{i=1}^m s_i \log(f_i) \]
is the \emph{log-likelihood function} or \emph{scattering potential}.
After choosing appropriate branches of the logarithm, the function
$\ell_\cA$ is well-defined on the complement
$\PP^{n-1}\backslash \bigcup_{f_i\in\cA}\{f_i=0\}$. 

For us, it is natural to assume $m > n$. With that hypothesis, the
complement of the arrangement 
is usually a very affine variety, i.e.~it is
isomorphic to a closed subvariety of an algebraic torus (see
e.g.~\cite{Lam}).  When the $f_i$ are linear forms, one recovers the
theory of hyperplane arrangements. This is included in our setup as an
important special~case.

In likelihood inference one wishes to maximize $\ell_\cA$ for given
$s_1,\dotsc, s_m$.  Due to the logarithms, the critical equations
$\nabla \ell_\cA = 0$ are not polynomial equations.
Of course, these rational functions can be made
polynomial by clearing denominators.  But, multiplying through with a high
degree polynomial is  a very bad idea in practice.  A key observation in this
paper is that the various modules of
(log)-derivations that have been considered in the theory of
hyperplane arrangements correctly solve the problem of
clearing denominators.

We now define graded modules over the polynomial ring $R$ which are
associated to~the arrangement $\mathcal{A}$.  To this end, consider
the following matrix with $m$ rows and $m+n$ columns:
\[
Q\,\, = \,\,\begin{bmatrix}
  f_1 & 0 	& \dots & 0 	& \frac{\partial f_1}{\partial x_1} & \dots & \frac{\partial f_1}{\partial x_n} \smallskip \\
  0 	& f_2 	& \dots & 0 	& \frac{\partial f_2}{\partial x_1} & \dots & \frac{\partial f_2}{\partial x_n} \smallskip \\
 \vdots     &  		& \ddots&  		& \vdots & & \vdots \smallskip \\
  0 	& 0 	& \dots & f_m 	& \frac{\partial f_m}{\partial x_1} & \dots & \frac{\partial f_m}{\partial x_n}
\end{bmatrix} \,\,\in\,\, R^{m\times (m+n)}.
\]
Each vector in the kernel ${\rm ker}(Q)$ is naturally partitioned
as
$ \begin{psmallmatrix}
  a\\
  b
\end{psmallmatrix}$,
where $a \in R^m$ and $b \in R^n$. With this partition,  let $
\begin{psmallmatrix}
  A\\
  B
\end{psmallmatrix}
\in R^{(m+n)\times l}$ be a matrix whose columns generate $\ker (Q)$.

\smallskip

We shall distinguish three graded $R$-modules associated with the
arrangement $\mathcal{A}$:
\begin{itemize}
% \item The \emph{Terao module} of $\mathcal{A} = \{f_1,\ldots,f_m\}$ is $\ker(Q)$. This is a submodule of $R^{m+n}$. \vspace{-0.1cm}
\item The \emph{Jacobian syzygy module} $\JSyz(\cA)$ is $\im (B)$. This is a submodule of $R^n$.
\vspace{-0.1cm}
\item The \emph{log-derivation module} $\logDer(\cA)$ is $\im (A)$. This is a submodule of $R^m$.
\vspace{-0.1cm}
\item The \emph{likelihood module} $M(\cA)$ is $\coker(A)$. This has $m$ generators and $l$ relations.
\end{itemize}
These modules are often identified since they are 
essentially isomorphic. See 
Lemma~\ref{lem:modules}.

\begin{example}[Braid arrangement] \label{ex:braid4}
Let $m=6, n = 4$ and let $\mathcal{A}$ be the graphic arrangement associated with the
complete graph $K_4$ (see Section \ref{sec:hyperplanes}). Writing $x,y,z,w$ for the variables, we have
\newcommand{\myScaleFactor}{1}  % Adjust this to <1 if the following
                                % two matrices are too big
\[
  Q \,\,=\,\,
\scalemath{\myScaleFactor}{\begin{bmatrix}
    x-y & 0 & 0 & 0 & 0 & 0 & 1 & -1 & 0 & 0 \\
    0 & x-z & 0 & 0 & 0 & 0 & 1 & 0 & -1 & 0 \\
    0 & 0 & x-w & 0 & 0 & 0 & 1 & 0 & 0 & -1 \\
    0 & 0 & 0 & y-z & 0 & 0 & 0 & 1 & -1 & 0 \\
    0 & 0 & 0 & 0 & y-w & 0 & 0 & 1 & 0 & -1 \\
    0 & 0 & 0 & 0 & 0 & z-w & 0 & 0 & 1 & -1 \\  
  \end{bmatrix}}. 
\]
The module ${\rm ker}(Q) \subset R^{10}$ is free. It is generated by the $l=4$ rows of the matrix
\[
\begin{bmatrix}
  A\\
  B
\end{bmatrix}^T
 = \, 
\scalemath{\myScaleFactor}{\begin{bmatrix}
  0 & 0 & 0 & 0 & 0 & 0 & -1 & -1 & -1 & -1 \\
  1 & 1 & 1 & 1 & 1 & 1 & -x & -y & -z & -w \\
  x\!+\!y & x\!+\!z & x\!+\!w & y\!+\!z & y\!+\!w & z\!+\!w & -x^2 & -y^2 & -z^2 & -w^2 \\
  x^2{+}xy{+}y^2 &\! \!x^2{+}xz{+}z^2 \! & \cdots   & \cdots & \cdots & \!\! z^2{+}zw{+}w^2 \! & -x^3 & -y^3 & -z^3 & -w^3 \\
\end{bmatrix}}\!. 
\]
The Vandermonde matrix in the last four columns represents the
syzygies on
$\,\nabla f = \bigl[ \partial f/\partial x , \partial f/\partial y ,
\partial f/\partial z , \partial f/\partial w \bigr]$, where $f$ is
the sextic $ (x-y)(x-z)(x-w)(y-z)(y-w)(z-w)$. This is the module
$\JSyz(\cA) \subset R^4$. The module $\logDer(\cA) \subset R^6$ is free of rank $3$ and generated by the
three nonzero rows of $A^T$.
This arrangement $\mathcal{A}$ has all the nice features in Section \ref{sec:3adjectives}.
\end{example}

Let $\Der_\CC(R)$ be the free $R$-module spanned by the partial
derivatives $\partial/\partial x_1,\dotsc, \partial/\partial x_n$.
Fix an arrangement $\cA$ as above and set $f = f_1 f_2 \cdots f_m$.
The {\em module of $\cA$-derivations} is
\begin{equation}\label{eq:Der}
\Der (\cA) \,\,= \,\,\set{\theta \in \Der_\CC (R) : \theta (f) \in \ideal{f}}.
\end{equation}
This definition is extensively used in the case of linear hyperplane
arrangements, but it makes sense for any homogeneous polynomial $f$.
The condition
$\theta(f) \in \ideal{f}$ ensures that the derivation
$\theta$, when applied to the log-likelihood $\ell_\cA$, yields an
honest polynomial rather than a rational function with $f_i$ in the
denominators.  This is expressed in Theorem~\ref{thm:evaluation}
via an injective $R$-module homomorphism
$\Der(\cA) \to R[s_1,\dotsc, s_m]$ which evaluates $\theta$
on~$\ell_{\cA}$.  

Using modules instead of ideals one can
store more refined information, namely how each $\theta \in \Der(\cA)$
acts on the individual factors $f_{i}$ or their logarithms.  While at first it
might seem natural to store elements of $\Der(\cA)$ as coefficient vectors
in~$R^n$, it is more efficient to store their values on
the~$f_i$.  This yields the log-derivation module~$\logDer(\cA)$, a
submodule of~$R^m$.  This representation has been used in computer
algebra systems like \Mtwo, together with the matrix $Q$
from above.  In the likelihood context, it appears in
\cite[Algorithm~18]{HKS}.

\begin{lemma}\label{lem:modules} Let $\cA$ be an arrangement in $\PP^{n-1}$, defined by
coprime polynomials $f_1,\ldots,f_m$.
\begin{enumerate}
% \item\label{it:modules1} The Terao module, the Jacobian syzygy module
% $\JSyz(\cA)$, the log-derivation module $\logDer(\cA)$, and
% the module of $\cA$-derivations $\,\Der(\cA)$ are
% all isomorphic as $R$-modules.
\item\label{it:modules1} The Jacobian syzygy module $\JSyz(\cA)$ and the module of $\cA$-derivations $\Der(\cA)$ are both isomorphic to $\ker(Q)$. Moreover, 
writing $\operatorname{Jac}(F)$ for the right hand side of $Q$, 
the Jacobian matrix of~$F = (f_1,\dotsc, f_m)$, 
there is a short exact sequence
  \[
    0 \rightarrow \ker(\operatorname{Jac}(F)) \rightarrow \ker(Q) \rightarrow \logDer(\cA) \rightarrow 0.
  \]
In particular, if $\ker(\operatorname{Jac}(F)) = 0$, then $\logDer(\cA)$ is isomorphic to the modules above.
\item\label{it:modules2} We have $\,\JSyz(\cA) \,\iso \,\JSyz_0(\cA) \oplus R\theta_E$,
where the second direct summand 
is the free rank $1$ module spanned by the \emph{Euler derivation}
\( \theta_E = \sum_{i=1}^n x_i \frac{\partial}{\partial x_i}\), and
\( \JSyz_0(\cA) = \ker (R^n \xrightarrow{\nabla f} R)\).
\item\label{it:modules3} The log-derivation module $\logDer(\cA)$ is isomorphic to the
first syzygy module of the likelihood module.  In particular, we have 
$\pd(M(\cA)) = \pd(\logDer(\cA)) + 1$.
\end{enumerate}
\end{lemma}

\begin{proof}
Since the $f_1,\dotsc, f_m$ are coprime, the condition
$\theta (f) \in \ideal{f}$ for a derivation $\theta$ is equivalent to $\,\theta (f_i) \in \ideal{f_i}$ for all $i=1,\dotsc, m$. 
These $m$ conditions correspond to lying in the individual kernels of the $m$ rows.  Hence $\Der(\cA)$ is isomorphic to $\ker(Q)$.  The module $\JSyz(\cA)$ is isomorphic to the kernel since the coefficients $a_i$ in $\theta(f_i) = a_i f_i$ are unique up to constants.  The surjection in the short exact sequence is computing these coefficients from~$\theta$.

Item \ref{it:modules2} is seen by writing
any element of $ \JSyz(\cA) \simeq \Der(\cA)$ as
$\theta = \theta' + \frac{1}{\deg f} \frac{\theta(f)}{f} \,\theta_E$.  Then
$\theta' = \theta - \frac{1}{\deg f} \frac{\theta(f)}{f} \,\theta_E$ satisfies
$\theta'(f) = 0$. Hence, $\theta' $ corresponds to an element in $ \JSyz_0(\cA)$.

For item \ref{it:modules3} we consider free resolutions over the ring $R$.
Let $A\in R^{m\times l}$ be the matrix whose image equals $\logDer(\cA)$.
A free resolution of $\coker (A)$ uses $A$ as the map $F_0\ot F_1$, i.e.
\[
0 \,\ot \,M(\cA) \,\ot \,R^m \,\xleftarrow{A} \, R^l \,\xleftarrow {A_2} \,F_2\, \ot\, \dotsb
\]
The image of $A$ is a submodule of $R^m$, and its free resolution
looks like this:
\[
0 \,\ot\,  \logDer(\cA)\, \xleftarrow{A} \, R^l \, \xleftarrow{A_2} \, F_2 \,\ot\, F_3\,\ot\, \dotsb
\]
The module $R^l$ sits in homological degree zero in the
resolution of $\im(A) = \logDer(\cA)$, and it sits in homological degree one in the resolution
of $\coker(A) = M(\cA)$.  The two resolutions agree from the map $A$ on to the
right, but the homological degree is shifted by one.
\end{proof}

Having introduced the various modules for an arrangement $\cA$, we now turn our attention
to likelihood geometry. This concerns the critical equations
$\nabla \ell_{\cA} = 0$ of the log-likelihood.  To capture the
situation for all possible data values $s_{i}$, one has the following
definition.
\begin{definition}\label{def:likelihoodCorr}
The {\em likelihood correspondence} $\cL_{\cA}$ is the Zariski closure in $\PP^{n-1}\times \PP^{m-1}$ of
\[
  \left\lbrace (x, s)\in \CC^{n}\times \CC^{m}\,\colon\,
    \frac{\partial \ell_\cA}{\partial x_i}(x,s)=0,\, i=1,\dots,n,\,
    f^s(x) \neq 0,\, F(x)\in X_{\reg} \right\rbrace,
\]
where $X$ is the Zariski-closure of the image of
$F\colon\CC^n\rightarrow\CC^{m},\, x \mapsto (f_1(x),\dots ,
f_m(x))$, and $X_{\reg}$ is its set of nonsingular points.  The
\emph{likelihood ideal} $I(\cA)$ is the vanishing ideal of~$\cL_\cA$.
\end{definition}
The likelihood correspondence is a key player in algebraic statistics
\cite{BCF, HuhSt}.  For example, the ML degree (see
Definition~\ref{def:MLdeg}) can be read off from the multidegree of
this variety.

\begin{lemma}\label{lem:liklihoodIPrime}
The likelihood ideal $I(\cA)$ is prime and $\cL_\cA$ is an irreducible
variety.
\end{lemma}
\begin{proof}
For each fixed vector $x \in \CC^{n}$, the likelihood
equations are linear in the $s$-variables.  The locus where
this linear system has the maximal rank is Zariski-open and dense in~$\CC^n$.
By our assumption $m>n$, the variety $\cL_\cA$ is therefore a vector bundle of rank $m-n$.
In particular, $\cL_\cA$ is irreducible, and its radical ideal 
$I(\cA)$ is prime.
\end{proof}

The second ingredient of Theorem~\ref{thm:twoIdeals} is
the Rees algebra of the likelihood module.
To define this object, we follow~\cite{SUV}.
Let $M$ be an $R$-module with $m$ generators.
The \emph{symmetric algebra of $M$} is a commutative $R$-algebra with
$m$ generators that satisfy the same relations as the
generators of~$M$.  More precisely, if $M = \coker (A)$ 
for some matrix $A\in R^{m\times l}$, then
\begin{equation}
\label{eq:symM}
\Sym (M) \,\,= \,\,R[s_1,\dotsc, s_m] \,/ \< \,(s_1,\dotsc, s_m) \, A \,\>.
\end{equation}
The \emph{Rees algebra $\cR(M)$ of~$M$} is the quotient of the symmetric
algebra $\Sym(M)$ by its $R$-torsion submodule.  Since $R$ is a
domain, its ring of fractions is a field and the likelihood module has
a rank.  This is the setup in~\cite{SUV} and $\cR(M)$ is a domain.
This can be shown, as in the case of ideals, by proving that its
minimal primes arise from minimal primes of~$R$.

\begin{definition}\label{def:ideals}
Let $\cA$ be an arrangement  and $M(\cA) = \coker (A)$ its likelihood module. 
The \emph{pre-likelihood ideal} of~$\cA$ is
$$
I_0(\cA) = \<\,(s_1,\dotsc, s_m)\,A\,\>\,\,\subset\,\, R[s] = \CC[x_1,\dotsc, x_n, s_1,\dotsc, s_m].$$
  This is the ideal shown in (\ref{eq:symM}),
where it presents the symmetric algebra of $M(\cA)$.

We further define $I$ as the kernel of the composition
\begin{equation}
\label{eq:composition}
 R[s] \, \to\, \Sym(M(\cA))\, \to\, \cR (M(\cA)). 
 \end{equation}
Thus, $I$ is an ideal in the ring on the left. It contains the pre-likelihood ideal  $I_0(\cA)$.
We refer to $I$ as the \emph{Rees ideal} of
the module~$M(\cA)$ because it
presents the Rees algebra of $M(\cA)$. \end{definition}

Theorem~\ref{thm:twoIdeals} states that the Rees ideal of $M(\cA)$ equals the likelihood ideal,
i.e.~$I  =  I(\cA)$. This will be proved below.
The ambient polynomial ring
$R[s] = \CC[x_1,\dotsc, x_n, s_1,\dotsc, s_m]$ is
 bigraded via $\deg(x_i) =
\begin{psmallmatrix}
  1\\
  0
\end{psmallmatrix}
$ for $i=1,\dotsc, n$ and $\deg(s_i) =
\begin{psmallmatrix}
  0\\1
\end{psmallmatrix}
$ for $i=1,\dotsc, m$.
% In the context of commutative algebra, Rees
% ideals often have a different, less symmetric presentation in a
% polynomial ring with as few variables as possible.
The Rees ideal can be computed with general methods in \Mtwo.
See~\cite{ETPS} for a computational introduction.  The output of the
general methods will differ from ours as these tools usually work with
minimal presentations of modules, thereby reducing the number of
variables~$s_{i}$.  For us it makes sense to preserve symmetry and
also accept non-minimal presentations.
% Here we aim for faster
% methods that are specialized to the cases of interest in algebraic
% statistics.

A module whose symmetric algebra agrees with the Rees algebra is
\emph{of linear type}.  This is the nicest case,
where the symmetric algebra has no $R$-torsion,
so it equals the Rees algebra.

\begin{definition}\label{def:gentle}
An arrangement $\cA$ is \emph{gentle} if its likelihood module is of
linear type, that is, if its likelihood ideal $I(\cA)$ equals
the pre-likelihood ideal $I_0(\cA)$. This happens if and only if
the map on the right in (\ref{eq:composition}) is an isomorphism, in which case 
$\Sym(M(\cA)) = \cR(M(\cA))$.
\end{definition}

\begin{example} The graphic arrangement of $K_4$ is gentle.
Fix the $ 6 \times 4$ matrix $A$ in Example~\ref{ex:braid4}.
The pre-likelihood ideal has three generators, one for each nonzero column of~$A$:
 \begin{equation}
\label{eq:likelihoodideal1} I_0(\mathcal{A}) \,\,=\,\,
\bigl\langle \,[ s_{12},s_{13},s_{14},s_{23},s_{24},s_{34} ] \cdot A \,\bigr\rangle \,\subset \, 
 R[s_{12},s_{13},s_{14},s_{23},s_{24},s_{34}].
 \end{equation}
One generator is $\sum_{ij} s_{ij}$. The other two generators
have bidegrees $\begin{psmallmatrix}
  1\\
  1
\end{psmallmatrix}$ and $
\begin{psmallmatrix}
  2\\
  1
\end{psmallmatrix}
$. Using \Mtwo, we find that the pre-likelihood ideal  $I_0 (\cA)$ is
prime. Hence, by Proposition~\ref{prop:localize} below, $I_0(\cA)$ equals the Rees ideal of $M(\cA)$, 
which is the likelihood ideal
$I(\cA)$. It defines a complete intersection in $\PP^3 \times \PP^5$.
This variety is the likelihood correspondence $\cL_{\cA}$.
\end{example}

\begin{example}[$n=3,m=4$] \label{ex:34notgentle}
The arrangement $\cA = \{ x, y, z, x^3+y^3+xyz \}$ is not gentle.
It consists of the three coordinate
lines and one cubic in $\PP^2$.
Its pre-likelihood ideal equals
\[ \begin{matrix} I_0(\cA) \, = \, \bigl\langle\,
 s_1+s_2+s_3+3 s_4,\,
      x z   \cdot s_2 - (3 y^2+x z) \cdot  s_3, \,
      y z  \cdot s_2 + (3 x^2+2 y z)  \cdot s_3 + 3 y z  \cdot s_4,  \\ \qquad \qquad
      (x^3+y^3)   \cdot s_2 \,+\, (3 y^3+x y z) \cdot s_3 \,+\, (3 y^3+x y z) \cdot  s_4\, \bigr\rangle.
\end{matrix}      
\]
This ideal is radical but it is not prime. Its prime decomposition equals
\[ \begin{matrix}
& I_0(\cA) & = & \!\! \bigl( I_0(\cA) + \langle x,y \rangle \bigr) \,\cap \, I(\cA), \qquad
{\rm where} \quad
I(\cA) \,=\, 
I_0(\cA) + \langle \,q\, \rangle  \smallskip \\
{\rm and} &
 q \! & = & z^2  \cdot s_2^2 
\,+\, z^2  \cdot  s_2 s_3 \,+\, 9 x y \cdot  s_3^2 \,-\, 2 z^2 \cdot  s_3^2 
\,+\, 3 z^2\cdot   s_2 s_4\, -\, 3 z^2  \cdot s_3 s_4.
\end{matrix}
\]
The extra generator $q$ of the likelihood ideal
is quadratic in the data vector $s = (s_1,s_2,s_3,s_4)$.
\end{example}

For hyperplane arrangements, our ideals were introduced
by Cohen et al.~\cite{CDFV} who called them the
\emph{logarithmic ideal} and the \emph{meromorphic ideal},
respectively.
In spirit of Terao's freeness conjecture, one can ask
whether gentleness is combinatorial, i.e.~can the underlying
matroid decide whether an arrangement is gentle?  
One candidate is the pair of non-isomorphic likelihood ideals
in \cite[Example 5.7]{DGS}. But this does not answer
our question, since all line arrangements in
$\PP^{2}$ are gentle (Theorem~\ref{thm:tlaag}).
A counterexample
must have rank at least~$4$.

Our technique for computing likelihood ideals of
arrangements rests on the following result. It transforms the
pre-likelihood ideal $I_0$ into the Rees ideal~$I$ via saturation.
\begin{proposition}\label{prop:localize}
Let $p$ be an element in $R$ such that
$M(\cA)[p^{-1}]$ is a free $R[p^{-1}]$-module.
 Then the likelihood ideal of the arrangement $\cA$ 
 is the saturation $\,I(\cA) = (I_0(\cA):p^\infty)$.
In particular, the arrangement $\cA$ is gentle if and only if its
pre-likelihood ideal $I_0(\cA)$ is prime.
\end{proposition}

\begin{proof}
The proof of the statement about $p$ uses the fact that the Rees algebra
construction commutes with localization. This can be found in
\cite[Section~2]{ETPS}.  The likelihood ideal $I(\cA)$ is always prime,
since the Rees algebra is a domain whenever $R$ is. Thus, if $I_0(\cA)$ is
not prime, then it is not the likelihood ideal and the arrangement
$\cA$ is not gentle.  If $I_0(\cA)$ is prime, then picking any suitable $p$
in the first part shows that it is the likelihood ideal.
\end{proof}

\begin{remark}
The existence of an element $p$ as in Proposition~\ref{prop:localize}
is guaranteed by generic freeness.  In our case, we can take $p$ as
the product of the $f_i$ and all maximal nonzero minors of the
Jacobian matrix of $F=(f_1,\dotsc, f_m)$. This follows from the
construction of the likelihood correspondence.  There
$F(x) \in X_{\reg}$ is required, but the proof of
Lemma~\ref{lem:liklihoodIPrime} requires only that the Jacobian of $F$
has maximal rank. We can replace $F(x) \in X_{\reg}$ by
this latter condition without changing the closure.
Computing the saturation tends to be a horrible computation.  For
practical purposes, it usually suffices to saturate $I_0$ at just a
few of these polynomials and checking primality after each step.  In
Example \ref{ex:34notgentle}, we can take $p$ to be any element in the
ideal $\langle x,y\rangle$ for the singular locus of
the cubic $x^3+y^3+xyz$.
\end{remark}

\begin{proof}[Proof of Theorem~\ref{thm:twoIdeals}]
Let $I$ be the prime likelihood ideal and $I_0$ the pre-likelihood
ideal of an arrangement~$\cA$.  Since the generators of $I_0$
vanish on the likelihood correspondence
$\cL_\cA$, we have $I_0 \subseteq I$.  Let $I'$ be the Rees ideal of
the likelihood module $M(\cA)$.  Clearly, also $I_0 \subseteq I'$ and
$I'$ is prime.  Let $p$ be an element as in
Proposition~\ref{prop:localize}, then
$I' = I_0 : p \subseteq I \colon p$.  Since $p \in R$ does not contain
any $s$ variables, $p\notin I$.  Hence, $I\colon p = I$ and thus
$I'\subseteq I$.  Conversely, also $I = I_0:f$ where $f$ equals a
sufficiently high power of the product of the polynomials cutting out
the singular locus of $X$ and the forms~$f_i$, another polynomial that
is $s$-free and no such polynomial vanishes on~$\cL_\cA$.  Hence, also
$I = I_0 : f \subseteq I' : f = I'$ and thus $I = I'$.
\end{proof}

We conclude this section with an emblematic result linking
 arrangements and likelihood.

\begin{theorem} \label{thm:evaluation}
  The evaluation of $\cA$-derivations at the log-likelihood function 
  \begin{equation*}
    \Der(\cA) \rightarrow I(\cA)_{(\ast,1)} \subset R[s],\quad 
    \theta \mapsto \theta(\ell_{\cA})
  \end{equation*}
  is an injective $R$-linear map onto $I_0(\cA)_{(\ast,1)}$. It is an isomorphism if and only if $\cA$ is gentle.
\end{theorem}

\begin{proof}
Any derivation $\theta$ maps $\ell_\cA$ to a rational function
in $\CC[s](x)$. The image is a polynomial in $\CC[x,s]$ if and only if $\theta \in \Der(\cA)$. Any polynomial in the image is linear in the $s$ variables, so the image lies in the $(\ast,1)$-graded part of $I(\cA)$.
The isomorphism between $\Der(\cA)$ and $\logDer(\cA)$
 in Lemma \ref{lem:modules}
ensures that the map is injective, and that
these polynomials generate the $(\ast,1)$-graded part of the ideal $I_0(\cA)$. If $\cA$ is gentle, then $I_0(\cA)_{(\ast,1)} = I(\cA)_{(\ast,1)}$.
\end{proof}

\section{Likelihood in statistics and physics}
\label{sec:likelihood}

Our study of hypersurface arrangements offers new tools for statistics
and physics.  We explain this point now.
% It is the aim of the present section to explain this
% point.
This happens in the general context of applied algebraic geometry
which is a rapidly growing field in the mathematical sciences.  In
applications, nonlinear models are ubiquitous, so it is not sufficient
to consider only arrangements of hyperplanes.

We start out with  basics on likelihood inference in
algebraic statistics~\cite{ABB,BCF,CHKS,HKS,HuhSt}.
Let $\cA$ be an arrangement in $\PP^{n-1}$,
given by homogeneous polynomials
$f_1,\ldots,f_m \in \RR[x_1,\ldots,x_n]$  of the same degree.
The  unknowns $x_1,\ldots,x_n$
are model parameters and the polynomials
$f_1,\ldots,f_m$ represent probabilities.  Let $X$ denote the
Zariski closure of the image of the map 
\[
  F\colon\CC^n\dashrightarrow\PP^{m-1},\, x\mapsto \bigl(f_1(x):f_2(x):\dots : f_m(x)\bigr).
\]
The algebraic variety $X$ represents a statistical model for
discrete random variables. Our model has $m$ states. The parameter region
consists of the points in $\RR^n$ where all $f_i$ are positive.
On that region, the rational function $\,f_i \,/ \sum_{j=1}^m f_j\,$ is the probability of observing the
$i$th state. In other words, the statistical model is given by the intersection of $X$ with the probability simplex $\Delta$ in $\PP^{m-1}$. Here, the $f_i$ are rarely linear, and the $s_i$ are
nonnegative integers which summarize the data.  Namely, $s_i$ is the
number of samples that are in state~$i$.

In statistics, one maximizes
the log-likelihood function $\ell_\cA$ over all points $x$ of the parameter region.
Here, the $s_{i}$ are given numbers and one considers the
critical equations $\nabla \ell_\cA = 0$. This is a system of rational function equations.
Any algebraic approach transforms these into polynomial equations.
Na\"ive clearing of denominators does not work because it introduces
too many spurious solutions.
The key challenge is to clear denominators
in a manner that is both efficient and mathematically sound.
That challenge is precisely the point of this paper.\par   

A key notion in likelihood geometry is the maximum likelihood degree, counting critical points of the likelihood function. We introduce a notion of this in our parametric arrangement setup. The likelihood correspondence $\cL_{\cA}$ lives in a product of projective spaces $\PP^{n-1} \times \PP^{m-1}$. Its class in the cohomology ring $H^*(\PP^{n-1}\times \PP^{m-1}; \ZZ) \cong \ZZ[p,u]/\langle p^n, u^m \rangle$ is a binary  form 
\begin{equation}
\label{eq:bidegree}
    \left[ \cL_{\cA} \right] \,\,=\,\, c_d p^d + c_{d-1} p^{d-1}u +  c_{d-2} p^{d-2} u^2 + \,\cdots\, +
    c_1 p u^{d-1} +  c_0 u^d,
\end{equation}
where $d = \mathrm{codim}(\cL_{\cA})$. This agrees with the \emph{multidegree} of $I(\cA)$ as in 
\cite[Part~II, \S 8.5]{MS}.

\begin{definition}
  \label{def:MLdeg}
  The \emph{maximum likelihood (ML) degree} $\mldeg(\cA)$ of the arrangement $\cA$ is the leading coefficient of $\left[ \cL_{\cA} \right]$, i.e.,~it equals $c_i$
   where $i$ is the largest index such that $c_i > 0$.
\end{definition}

If $ c_d > 0$ then $\mldeg(\cA) = c_d$ and
 Definition \ref{def:MLdeg} gives a critical point count.

\begin{proposition}
  \label{prop:MLDegFromMultiDeg}
If $\,\mldeg(\cA) = c_d$ then the set
\begin{equation}
    \label{eq:criticalPointsLikelihood}
    \left\{ x\in \PP^{n-1} \,\colon\, \nabla\ell_{\cA}(x,s)=0,\,
    f^s(x) \neq 0,\, F(x)\in X_{\reg} \right\},
\end{equation} 
is finite for generic choices of $s$. Its cardinality equals $\mldeg(\cA)$ and does not depend on $s$.
\end{proposition}

\begin{proof}
  Under the assumption $c_d>0$, the projection $\pi\,\colon\, \cL_{\cA} \rightarrow \PP^{m-1}$ is finite-to-one. A general fiber has cardinality $c_d$ and is described by \eqref{eq:criticalPointsLikelihood}.
  \end{proof}

\begin{remark}
  \label{rem:implicitVsParametricMLdeg}
  The above setup differs from the one common to algebraic statistics in several aspects:
  First, ``generic choices of $s$'' means generic in a subspace of
  $\CC^m$. This is usually  $\{s: \sum_{i=1}^m d_i s_i = 0\}$.
   Second, Proposition \ref{prop:MLDegFromMultiDeg} gives a \emph{parametric} version of the ML degree, whereas \cite{BCF, HKS, HuhSt} define the ML degree \emph{implicitly}. Moreover, in \cite{CHKS}, the hypersurface defined by $\sum_{i=1}^{m}f_i$ is added to the arrangement. Only this modification allows the interpretation of $\mathcal{A}$
     as a statistical model, as described in the paragraph above. If this hypersurface is included in $\cA$ and we assume that the parametrization is finite-to-one, then our parametric ML degree is an integer multiple of the implicit ML degree. Under these assumptions, there is a flat morphism from the parametric to the implicit likelihood correspondence in \cite{HuhSt}. The induced map on Chow rings is injective, and the claim follows. Our definition via the multidegree of $\cL_{\cA}$ allows for a sensible notion even in the case where the parametrization is not finite-to-one. This appears for example in the formulation of
     toric models given below.
\end{remark}

For illustration
% We now illustrate the relationship between our hypersurface arrangement setup and the statistics setup.
% To this end,
we revisit
the \emph{coin model} from the introduction of \cite{HKS}. 

\begin{example} \label{ex:gambler}
A gambler has two biased coins,
one in each sleeve, with unknown biases~$t_2,t_3$.  They select one of
them at random, with probabilities $t_1$ and $1-t_1$, toss that
coin four times, and record the number of times heads comes up. If
$p_i$ is the probability of $i-1$ heads then
\newcommand{\myScaleFactor}{1}
\begin{equation}
\label{eq:coins}
\scalemath{\myScaleFactor}{
  \begin{matrix}
    p_1 & = & t_1 \cdot (1-t_2)^4 & + & (1-t_1) \cdot (1-t_3)^4, \\
    p_2 & = & 4 t_1 \cdot t_2 (1-t_2)^3 & + & 4 (1-t_1) \cdot t_3 (1-t_3)^3 ,\\
    p_3 & = & 6 t_1 \cdot t_2^2 (1-t_2)^2 & + & 6 (1-t_1) \cdot t_3^2 (1-t_3)^2, \\
    p_4 & = & 4 t_1  \cdot t_2^3 (1-t_2) & + & 4 (1-t_1) \cdot t_3^3 (1-t_3), \\
    p_5 & = & t_1 \cdot t_2^4 & + & (1-t_1) \cdot t_3^4 .
  \end{matrix}
}
\end{equation}
We homogenize
by setting $t_j = x_j/x_4$ for $j\in\{1,2,3\}$. 
Let $f_i(x)$ be the numerator of $p_i(t)$
after this substitution.  This is a homogeneous polynomial in four
variables of degree $d_i  = 5$.  We finally set
$f_6(x) = x_4$ and $d_6=1$.  If we now take
$s_6 = - d_1 s_1 - d_2 s_2 - \cdots - d_{5} s_{5}$, then we are in
the setting of Section~\ref{sec:arrangements}.
Namely, we have an arrangement $\mathcal{A}$ of $m=6$ surfaces in $\PP^3$.

We observe $N$ rounds of this game, and we record the outcomes in the
data vector $(s_1,s_2,s_3,s_4,s_5) \in \NN^5$, where $s_i$ is the
number of trials with $i-1$ heads.  Hence, $\sum_{i=1}^5 s_i = N$.
Our assignment $s_6 = - 5N$ ensures that
$d_1 s_1 + \cdots + d_6 s_6$ lies in $I_0(\cA)$.
The task in statistics is to learn the parameters $\hat t_1, \hat t_2, \hat t_3$
from the data $s_1,\ldots,s_5$.
% using likelihood inference.
%
The ML degree
% of this model
is~$24$.
Indeed, the equations $\nabla\ell_{\cA}(x,s)=0$ have $24$ complex solutions~$x =
(t,1) \in \PP^4$ for random
data $s_1,s_2,s_3,s_4,s_5$, provided
$t_1 (1-t_1) (t_2 - t_3) \not= 0$.  In \cite{HKS} it is reported that
the ML degree for this model is $12$.  This factor two arises because of the
two-to-one parametrization~\eqref{eq:coins}.
% the number there is computed for the implicit model and the
% parametrization  is .

In summary, our projective formulation realizes
 the coin model as an arrangement $\cA$ in
$\PP^3$ with $n=4, m=6$, and $d_1=d_2 = d_3 = d_4 = d_5 = 5$ and $d_6=1$.  The quintics $f_1,f_2,f_3,f_4,f_5$ have $13, 12, 9, 6, 3$ terms
respectively. For instance, the homogenization of $p_4(t)$ yields
\[
f_4(x) \,=\, 4(- x_1x_2^4 +  x_1x_3^4 + x_1x_2^3x_4 - x_1x_3^3x_4 - x_3^4x_4 + x_3^3x_4^2) .
\]
The pre-likelihood ideal $I_0(\cA)$ has six generators, of
bidegrees $
\begin{psmallmatrix}
  0\\
  1
\end{psmallmatrix},
\begin{psmallmatrix}
  2\\
  1
\end{psmallmatrix},
\begin{psmallmatrix}
  10\\
  1
\end{psmallmatrix}
$, and $
\begin{psmallmatrix}
  13\\
  1
\end{psmallmatrix}
$ thrice.  The first
ideal generator is $5 (s_1 + s_2 + s_3 + s_4 + s_5) + s_6$, and the
second ideal generator is
\[ 4 s_6( x_1  x_2 -  x_1  x_3 +  x_3  x_4)\, +\,
 5(s_2 + 2 s_3 + 3 s_4 + 4 s_5) x_4^2 . \]
 We invite the reader to test whether
     $\mathcal{A}$ gentle. Is
 $I_0(\cA)$ equal to the likelihood ideal $I(\cA)$?
\end{example}

We now turn to the  two-parameter models
on four states seen in the Introduction of \cite{CHKS}.

\begin{example} \label{ex:fourconics}
Let $n=3$, $m=5$, $\,d_1=d_2=d_3=d_4 = 2$,
and $d_5 = 1$. This gives
arrangements of four conics and the line at infinity in $\PP^2$.
One very special case is the
independence model for two binary random variables,
in a homogeneous formulation:
\[
  f_1 = x_1 x_2,
  f_2 =  (x_3-x_1)x_2,
  f_3 =x_1 (x_3-x_2),
  f_4 =  (x_3-x_1)(x_3-x_2),
  f_5 = x_3.
\]
The arrangement is tame and free (see Section \ref{sec:3adjectives}), but not gentle; the pre-likelihood ideal is
\[
    \langle s_+ ,\,s_5,\,x_3 \rangle \,\cap \,
    \langle\, 2 s_+ + s_5,\,
    s_+\, x_1 - (s_1 \!+\! s_3) \,x_3,\,
    s_+ \,x_2 - (s_1 \!+\! s_2) \,x_3,\,
    (s_1 \!+\! s_2) \,x_1 - (s_1\!+\! s_3)\, x_2  
    \rangle.
\]     
Here $s_+ = s_1 + s_2 + s_3 + s_4$ is the sample size. The likelihood ideal
is the second intersectand. Its four generators
confirm that the ML degree equals $1$. The likelihood ideal is not a complete intersection since $\codim(I)=3$.
For the implicit formulation see \cite[Example 2.4]{BCF}.

As in the Introduction of \cite{CHKS}, we compare this
with arrangements given by
random ternary quadrics $f_1,f_2,f_3,f_4$ plus $f_5 = x_3$.
Such a generic arrangement is
tame and gentle. The likelihood ideal equals the pre-likelihood ideal.
It is minimally generated by seven polynomials: the linear
form $2(s_1+s_2+s_3+ s_4) + s_5$,  four generators
of degree $
\begin{psmallmatrix}
  6\\
  1
\end{psmallmatrix}
$, and two generators of degree $
\begin{psmallmatrix}
  7 \\1
\end{psmallmatrix}
$.
The bidegree (\ref{eq:bidegree}) of the likelihood correspondence 
$\mathcal{L}_\cA \subset \PP^4 \times \PP^2$
equals $25 p^2 + 6 pu + u^2$.
Hence, the ML degree equals $25$,
as predicted by \cite[Theorem 1]{CHKS}.
\end{example}

\smallskip

In algebraic statistics, a model is called {\em toric}
if each probability $p_i$  is a monomial in the model parameters. It is represented by a toric variety $X_A$, the image closure of a map
\[
    \phi_A \colon (\CC^*)^n \rightarrow \PP^N,\quad (x_1,\dots, x_n) \mapsto (x^{a_0}:\dots :x^{a_N}),
\]
where $A$ is an integer matrix of size $n\times (N+1)$ with columns
$a_0,\dots,a_N$.  By \cite{HuhVeryAffine}, the ML degree of $X_A$ is
the signed Euler characteristic of $X_A\backslash \mathcal{H}$, where
$\mathcal{H}$ is the hyperplane arrangement given by
$\{y_0,\dots,y_N,y_0+\dots +y_N\}$ in which the $y_i$ are the
coordinates of~$\PP^N$.

Let $f = x^{a_0}+\dots +x^{a_N}$ be the coordinate sum. Assuming that the map $\phi_A$ is one-to-one, it gives an isomorphism of very affine varieties between $\{x\in (\CC^*)^n \mid f(x) \neq 0\} 
$ and $X_A\backslash \mathcal{H}$. Its signed Euler characteristic is equal to the number of critical points of the function 
\begin{equation}
    \label{eq:likelihoodToric}
    x_1^{s_1}x_2^{s_2}\dots x_n^{s_n} f^{s_{n+1}}, 
\end{equation}
for generic values $s_1,\dots,s_n$ and $s_{n+1}=-\frac{1}{d}(s_1+\dots +s_n)$, where $d = \deg(f)$. We can encode this in the arrangement setup by setting $f_i = x_i$ for $i=1,\dots,n=m-1$ and $f_m = f$. The likelihood function of this arrangement $\cA = \{x_1,\dots,x_n,f\}$ agrees with \eqref{eq:likelihoodToric}. The ML degree of $X_A$ is equal to the ML degree of $\cA$. In 
situations where $\phi_A$ is not one-to-one, the ML degree of $\cA$ is a product of the degree of the fiber with the ML degree of $X_A$.\par   

% For the convenience of the reader more akin to the arrangement literature, we would like to emphasize that this differs from the notion of a toric arrangement in the sense of e.g.\ \cite{toricArrangements}.

One instance with $n=3$ was seen in
Example \ref{ex:34notgentle}.
Our representation of a toric model depends on the
choice of the parametrization and so does gentleness of the
arrangement~$\cA$.  This is one reason why previous
work on likelihood geometry emphasized the implicit
representation. We illustrate the toric setup with
the most basic model  in algebraic statistics.

\begin{example}[Independence]
The independence model for two binary random variables is
\[
p_{00} = a_0 b_0 , \,\,
p_{01} = a_0 b_1,\,\,
p_{10} = a_1 b_0,\,\,
p_{11} = a_1 b_1.
\]
This parametrizes the Segre surface
$\{p_{00} p_{11} = p_{01} p_{10} \}$  in $\PP^3$.
This model is known to have ML degree $1$. The four conics formulation
of this model given in Example~\ref{ex:fourconics} was not gentle.

We can represent this independence model as a 
toric model  by setting $n=4$ and
\[
\mathcal{A} \, = \, \{\,a_0,a_1,b_0,b_1, \,f\, \} \quad
\hbox{with} \,\,f = a_0 b_0 + a_0 b_1 + a_1 b_0 + a_1 b_1 .
\]
This is a gentle arrangement of $m=5$ surfaces in $\PP^3$. Its likelihood ideal equals
\[
I(\cA) \,=\, I_0(\cA) \,=\, \bigl\langle
\,s_1+s_2+s_5,\,
s_3+s_4+s_5,\,
(b_0 +b_1) s_4 + b_1 s_5,\,
(a_0+a_1) s_2 + a_1 s_5\,
\bigr\rangle
\]
The arrangement $\mathcal{A}$  is an overparametrization.
A minimal toric model would live
in the plane $\PP^2$. For instance,
$ \mathcal{A}' \,\,=\,\, \{\, x,y,z,\, x y+xz+yz+z^2 \,\} $.
This arrangement is also gentle. Its multidegree is
$p^2 u + 2 p u^2 + u^3$.
One can compute $I_0(\cA') = I(\cA')$ as shown in Section \ref{sec:software}.
\end{example}

\smallskip

We finally turn to \emph{scattering equations} in particle physics.  In the
CHY model~\cite{CHY} one considers scattering equations on the moduli
space $\mathcal{M}_{0,n}$ of $n$ labeled points in $\PP^1$.  The
scattering correspondence appears in \cite[eqn~(0.2)]{Lam}, and is
studied in detail in \cite[Section~3]{Lam}.  The formulation in
\cite[eqn~(3)]{ST} expresses the positive region $\mathcal{M}^+_{0,n}$ of $\mathcal{M}_{0,n}$ as a linear
statistical model of dimension $n\!-\!3$ on
$n(n\!-\!3)/2$ states. Adding another coordinate for the homogenization, we have $m = \binom{n-1}{2}$ in our setup.
The ML degree equals $(n-3)!$.  If the data $s_1,\ldots,s_m$ are real,
then all $(n-3)!$ complex critical points are real by Varchenko's
Theorem \cite[Proposition~1]{ST}.  The case $n=6$ is worked out
% explicitly
in \cite[Example 2]{ST}.  This model has $m-1=9$ states and
the ML degree is $6$.  The nine probabilities $p_i$ are given in
\cite[eqn~(6)]{ST}.  These $p_i$ sum to $1$ and all six critical
points in \cite[eqn~(9)]{ST} are real. \par 
Usually, we think of $\mathcal{M}_{0,n}$ as the set of points for which the $2\times 2$ minors of the matrix 
\[
    \begin{bmatrix}
      0 & 1 & 1 & \dots & 1 & 1 \\
      -1 & 0 & y_1 & \dots & y_{n-3} & 1 \\ 
    \end{bmatrix}
\]
are non-zero. If we homogenize the resulting equations by considering the $2\times 2$ minors of 
\[
    \begin{bmatrix}
      0 & 1 & 1 & \dots & 1 & 1 \\
      -1 & x_1 & x_2 & \dots & x_{n-2} & x_{n-1} \\ 
    \end{bmatrix},
\]
then  $\mathcal{M}_{0,n}$ becomes the complement of the braid
arrangement. This is the graphic arrangement of $K_{n-1}$ (see Section
\ref{sec:hyperplanes}), defined by the $\binom{n-1}{2}$ linear forms
$x_i-x_j$ for $1\le i<j\le n$.

For example, $\mathcal{M}_{0,5}$ can be viewed as the complement of
the arrangement in Example~\ref{ex:braid4}.  In this case, the image
of the likelihood correspondence in $\PP^2 \times \PP^5$ under the map
to data space $\PP^5$ is the hyperplane
$\{s_{12} + s_{13} + s_{14} + s_{23} + s_{24} + s_{34} = 0\}$.  This
map is $2$-to-$1$. By \cite[Section~2]{ST}, the fibers are the two
solutions to the scattering equations in the CHY model for five
particles.
% The CHY model $\mathcal{M}_{0,n+2}$ coincides with the complement of
% the graphic arrangement associated with the complete graph $K_{n+1}$,
% i.e., the braid arrangement. Indeed, if we homogenize the polynomials
% in \cite[eqn (6)]{ST} and we add the new variable $x_4$, then we
% obtain the ten hyperplanes $x_i - x_j$ for $ 1 \leq i < j \leq 5$. For
% a perfect match, it helps to set $x_5 = 0$ and ignore to the scaling
% factors $5/9$, $1/3$ and $1/9$.  
A similar identification works for
every graphic arrangement, when some edges of $K_{n-1}$ are deleted. Physically, this corresponds to setting some Mandelstam invariants to zero.  The article \cite{EPS}  studies
graphic arrangements of ML degree one from a physics perspective.  For
instance, in \cite[Example~1.3]{EPS}, we see $K_5$ with three edges
removed.

\section{Gentle, free and tame arrangements}
\label{sec:3adjectives}

\epigraph{I was tame, I was gentle 'til  \\ the circus life made me mean.}{Taylor Swift}

The concept of freeness has received considerable attention in the theory of
hyperplane arrangements, see e.g.~\cite[Theorem 4.15]{orlik2013arrangements}.
Also, the notion of tameness \cite[Definition~2.2]{CDFV} appeared in this context. In this section we explore
the relationship between these concepts and the gentleness of an arrangement. We shall 
explain the following (non)implications:
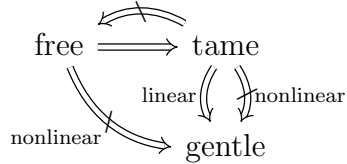
\begin{center}
  \begin{tikzcd}
    \text{free} \arrow[r, Rightarrow] \arrow[rd, "\text{nonlinear}"', Rightarrow, bend right, "\boldsymbol{\backslash}"{anchor=center,sloped}] & \text{tame} \arrow[l, Rightarrow, bend right, "\boldsymbol{\backslash}"{anchor=center,sloped}] \arrow[d, "\text{linear}"' {yshift=-0.5em}, Rightarrow, bend right] \arrow[d, "\text{nonlinear}" {yshift=-0.5em}, "\boldsymbol{\backslash}"{anchor=center,sloped},  Rightarrow, bend left]\\
    & \text{gentle}         
  \end{tikzcd}
\end{center}

\begin{definition}
A hypersurface arrangement $\cA$ is \emph{free} if $\logDer(\cA)$ is a free
$R$-module.
\end{definition}

By Lemma~\ref{lem:modules}, $\cA$ is free if and only if the
likelihood module $M(\cA)$ has projective dimension one. Let $\Omega^1(\cA) = \Hom(\Der(\cA), R)$ be the module of logarithmic differentials with poles along $\cA$. Nonstandard, but justified by
\cite[Proposition~2.2]{denham2009complexes}, we define
\[
\Omega^p(\cA) \,\,=\,\, \left( \bigwedge\nolimits^p \Omega^1(\cA) \right)^{\! \vee\vee}.
\]

\begin{definition} \label{def:tame} A hypersurface arrangement $\cA$ is \emph{tame} if
\[
\pd_R(\Omega^p(\cA))\, \leq\, p \quad \text{for all } \,\,0\leq p\leq r(\cA),
\]
where $r(\cA)$ is the smallest integer such that $\Omega^p(\cA) = 0$ for all $p > r(\cA)$.
\end{definition}

Clearly, every free arrangement is tame. The braid arrangement from Example \ref{ex:braid4} is free. We have already seen that the braid arrangement is also gentle. This holds more generally.

\begin{theorem} \label{thm:tlaag}
  Tame linear arrangements are gentle.
\end{theorem}
  
\begin{proof}
  The statement follows from \cite[Corollary 3.8]{CDFV} and Proposition \ref{prop:localize}.
  The ideal $I$ in \cite{CDFV} is our pre-likelihood ideal $I_0(\cA)$, and their variety
  $\overline{\Sigma}$ is our likelihood correspondence  $\mathcal{L}_\cA$.
\end{proof}
 
In the linear case, the modules $\Omega^p(\cA)$ are reflexive and hence of projective dimension at most $r(\cA)-2$.
Therefore, every linear arrangement in $\PP^2$ is tame and thus also gentle.
Although freeness is a strong property for an arrangement, for
hypersurfaces it does not necessarily imply gentleness. We saw a free
arrangement that is not gentle in Example \ref{ex:fourconics}.  We do
not know whether the reverse implication ``gentle $\Rightarrow$ tame''
holds. To the best of our knowledge, this is unknown even for the
linear case; see the Introduction of \cite{CDFV}.

\begin{problem}
\label{prob:gentle_tame}
  Is every gentle arrangement tame?
\end{problem}

For a linear arrangement, freeness is equivalent to the (pre-)likelihood ideal being a complete intersection \cite[Theorem~2.13]{CDFV}. As Example \ref{ex:fourconics} shows, this is not necessarily true in the hypersurface case. However, under the additional assumption that $\cA$ is gentle, we can generalize \cite[Theorem~2.13]{CDFV}. This connects to \cite{HuhSt} where the authors ask for a
characterization of statistical models whose likelihood ideal is a complete intersection.
\begin{theorem}

\label{thm:freeCI}
Let $\cA$ be a gentle arrangement of hypersurfaces. Then $\cA$ is free if and only if
the likelihood ideal $I(\cA)$ is a complete intersection.
\end{theorem}
The proof uses the notion of relative modules of logarithmic differential forms.
For a ring $S$ and an $S$-algebra $T$, $\Omega^1_{T/S}(\cA)$ denotes the $T$-module of $S$-valued Kähler differentials with poles along~$\cA$. For $S=\CC$ and $T=R$, one recovers the modules $\Omega^1(\cA)$ from Section \ref{sec:3adjectives}.

\begin{proof}
Suppose $\cA$ is free of rank $l$, i.e.~the 
% log-derivation 
module $\logDer(\cA)$ is a free module
generated by $\left\{ D_1,\dots,D_l \right\}$, the 
columns of the matrix $A$ from Section~\ref{sec:arrangements}.
 Consequently, the pre-likelihood
ideal $I_0(\cA)$ has $l$ generators. By assumption, the arrangement $\cA$ is gentle, and hence
$I_0(\cA) = I(\cA)$. Since $\cL_{\cA}$ has codimension $l$, this shows that
$I(\cA)$ is a complete intersection.\par
Conversely, assume $I(\cA)$ has $l$ generators $g_1,\dots,g_l$. 
Similarly to Theorem \ref{thm:evaluation},
for $1\leq i \leq l$, let $\theta_i \in \Der_S(\cA)$ be a derivation
for which $\theta_i(\ell_{\cA}) = g_i$. Here,
$S = \CC[s_1,\dots,s_m]$ and $\Der_S(\cA)$ is the module of $S$-linear logarithmic derivations on $S\otimes_{\CC} R$. The module $\Der_S(\cA)$ is generated by the $\theta_i$ and has rank $l$, hence it is free. By extension of scalars,
\[
    \Omega^1_{R/\CC}(\cA) \otimes_{R} (S\otimes_{\CC} R) \,\cong \,\Omega^1_{S\otimes R/ S}(\cA),
\]
and $\Omega^1_{S\otimes R/ S}(\cA)$ is dual to $\Der_S(\cA)$. Then, by tensor-hom adjunction, we obtain 
\begin{align*}
  \Der_S(\cA) & \,\cong \,\Hom_{S\otimes R}(\Omega^1_{R/\CC}(\cA) \otimes_{R}(S\otimes_{\CC} R),\, S\otimes_\CC R) \\ & \,\cong \,\Hom_R(\Omega^1_{R/ \CC}(\cA),\, \Hom_{S\otimes R}(S \otimes_\CC R,\, S \otimes_\CC R))\\
  & \,\cong \,\Hom_R(\Omega^1_{R/ \CC}(\cA),\, S \otimes_\CC R).
\end{align*}
Since $\Omega^1_{R/\CC}(\cA) = \Omega^1(\cA)$ is finitely presented and $S\otimes_\CC R$ is faithfully flat, it follows that $\Der(\cA) = \Hom(\Omega^1(\cA),\, R)$ is free, so $\cA$ is a free arrangement.
\end{proof}

In the case of a free and gentle arrangement, it is now easy to read off the ML degree.

\begin{corollary} \label{cor:MLdegproduct}
Let $\cA$ be free and gentle. If the columns of $A$ have degrees $d_1,\ldots,d_l$ then
\begin{equation}
\label{eq:MLdegCI}
\mldeg(\cA) \,\,= \prod_{i\,:\,d_i > 0} d_i.
\end{equation}
\end{corollary}

\begin{proof}
By definition, the ML degree is the
leading coefficient in the multidegree of $I(\cA)$. Since $\cA$ is
free and gentle, by Theorem~\ref{thm:freeCI}, the likelihood ideal is
a complete intersection,  and  it is linear in the $s$ variables. Therefore, the
cohomology class in (\ref{eq:bidegree}) is the product
\[
\left[ \cL_{\cA} \right]\, \,=\, \prod_{i=1}^{r(\cA)} \left( d_i p + u \right).
\]
Our assertion now follows because (\ref{eq:MLdegCI})
is the leading coefficient of this binary form.
\end{proof}
  
\begin{example}
For the braid arrangement in Example~\ref{ex:braid4}, the matrix
$A^T$ has two rows of positive degree.
Hence, by \eqref{eq:MLdegCI}, $\mathrm{MLdeg}(\cA) = 1\cdot 2 = 2$.
For general $n$, the braid arrangement $\cA(K_n)$ has ML
degree $(n-3)!$, as stated in our physics discussion about $\mathcal{M}_{0,n}$
in Section \ref{sec:likelihood}.
\end{example}

Symmetric algebras and Rees algebras are ubiquitous in commutative
algebra.  Many papers studied them, especially when $M$ has a short
resolution.  The \emph{Fitting ideals} of $M$~play an essential role.
Let $I_t(A)$ be the ideal generated by the $t\times t$-minors of a
matrix $A \in R^{m\times l}$ with $M= \coker (A)$.  These ideals are
independent of the presentation of~$M$ \cite[Section~20.2]{Eis}.

Early work of Huneke \cite[Theorem~1.1]{Hun} characterizes when the
symmetric algebra of a module $M$ with $\pd(M) = 1$ is a domain, and
thus when a free arrangement is gentle.  This happens if and only if
$\depth (I_t(A), R) \ge \rk(A) + 2 - t$ for all $t=1,\dotsc, \rk(A)$.
Huneke also showed that in this case the symmetric algebra is a
complete intersection, one direction of our Theorem~\ref{thm:freeCI}.
Simis and Vasconcelos~\cite{SV81} obtained similar results
concurrently.

In the 40+ years since these publications, many variants have been
found. For example, authors studied for which $k$ all inequalities
$\depth (I_t(A)) \ge \rk(A) + (1+k) - t$ hold.  If this is the case,
then $M$ is said to have \emph{property~$\mathcal{F}_k$}.  Assuming
$\mathcal{F}_k$ and related hypotheses, properties
(e.g.~Cohen--Macaulay) of symmetric and Rees algebras of modules were
studied.

A notable special case arises if the double dual 
$M^{\vee\vee}$ of a module $M$  is free.  In
\cite[Section~5]{SUV} such an $M$ is called an \emph{ideal module}
because it behaves very much like an ideal.  
Every ideal module $M$
is the image of a map of free modules, and various criteria for
gentleness (i.e.\ linear type) of $M$ can be derived. These might give rise to more efficient computational tests for gentleness. For example, the likelihood module of the octahedron in
Example~\ref{ex:octahedron1} is an ideal module.
In conclusion, we invite commutative algebraists to
join us in exploring the likelihood geometry of arrangements,
and its applications ``in the sciences''.

\section{Graphic arrangements}
\label{sec:hyperplanes}

Graphic hyperplane arrangements are a mainstay of combinatorics.  They 
are subarrangements of the braid arrangement.
In particle physics \cite{EPS, Lam} they arise from the moduli space $ \mathcal{M}_{0,n}$.
 Fix the 
polynomial ring $R = \CC[x_1,\dotsc, x_n]$, and let $G = (V,E)$ be an
undirected graph with vertex set $V = \{1,\ldots, n\}$.
The \emph{graphic arrangement} $\cA(G)$
consists of the hyperplanes $\set{ x_i - x_j : \left\{ i,j \right\} \in E}$. 
This arrangement lives in $\PP^{n-1}$, but can also be viewed in the space $\PP^{n-2} = \mathrm{Proj}(\CC[x_i - x_n \,:\, i=1,\dots,n-1])$. We mostly consider the latter in this section.
%obtained by projecting from the point $(1:1:\cdots:1)$ which lies in all hyperplanes.

A classical result due to Stanley, Edelman and Reiner states that
$\cA(G)$ is free if and only if the graph $G$ is chordal (see
\cite{AKMM} for further developments).  The complete graph $G = K_4$
% in Example \ref{ex:braid4}
is chordal and we saw that
$\logDer(\cA({K_4})) \simeq R^3$.
The octahedron in Example~\ref{ex:octahedron1} is not chordal.

In this section, we examine the notion of gentleness for graphic arrangements.
A priori, it is not clear that there exist graphs whose arrangement is not gentle.
We now show~this.

\begin{example}[Octahedron]\label{ex:octahedron1}
Consider the graph $G_{\text{oct}}$ of an octahedron, depicted in Figure~\ref{fig:oct}.
Let $R = \QQ[x_1,\dotsc, x_6]$.  The graphic
arrangement $\cA(G_{\text{oct}})$ consists of the  $12$ hyperplanes
\[
x_1-x_2, x_1-x_3, x_1-x_5, x_1-x_6, x_2-x_3, x_2 -x_4, x_2 - x_6,
x_3-x_4, x_3-x_5, x_4-x_5, x_4-x_6, x_5-x_6.
\]
The likelihood module has $12$ generators and $6$ relations,
of degrees one, two and three
(4 times), plus the Euler relation of degree
zero.  These relations correspond to the 7 generators of the
pre-likelihood ideal $I_0$.  A computation with \Mtwo shows that
\( I_0 : (x_1-x_2) \neq I_0 \).

\begin{figure}[htpb]
\centering
\begin{tikzpicture}[line join=bevel,z=-5.5,scale=1.5]
	\coordinate (A1) at (0,0,-1);
	\coordinate (A2) at (-1,0,0);
	\coordinate (A3) at (0,0,1);
	\coordinate (A4) at (1,0,0);
	\coordinate (B1) at (0,1,0);
	\coordinate (C1) at (0,-1,0);
	
	\draw (A1) -- (A2) -- (B1) -- cycle;
	\draw (A4) -- (A1) -- (B1) -- cycle;
	\draw (A1) -- (A2) -- (C1) -- cycle;
	\draw (A4) -- (A1) -- (C1) -- cycle;
	\draw [fill opacity=0.7,fill=gray!80!black] (A2) -- (A3) -- (B1) -- cycle;
	\draw [fill opacity=0.7,fill=gray!80!black] (A3) -- (A4) -- (B1) -- cycle;
	\draw [fill opacity=0.7,fill=gray!30!black] (A2) -- (A3) -- (C1) -- cycle;
	\draw [fill opacity=0.7,fill=gray!70!black] (A3) -- (A4) -- (C1) -- cycle;
	
	\node at (1.2, 0, 0) {4};
	\node at (-1.2, 0, 0) {1};
	\node at (0, 1.2, 0) {2};
	\node at (0, -1.2, 0) {5};
	\node at (-0.1, -0.1, 1.3) {3};
	\node at (0.1, 0.1, -1.2) {6};
\end{tikzpicture} \qquad \qquad
\begin{tikzpicture}[scale=1.7, every node/.style={circle, draw, fill=black, inner sep=2pt}]
\node[label=below left:$1$] (1) at (-1,0) {};
\node[label=above left:$5$] (4) at (-1/2,{sqrt(3)/2}) {};
\node[label=above right:$6$] (6) at (1/2,{sqrt(3)/2}) {};
\node[label=right:$4$] (3) at (1,0) {};
\node[label=below right:$2$] (2) at (1/2,{-sqrt(3)/2}) {};
\node[label=below left:$3$] (5) at (-1/2,{-sqrt(3)/2}) {};
\draw[thick] (1) -- (2) -- (3) -- (4) -- (1);
\draw[thick] (1) -- (5) -- (2);
\draw[thick] (1) -- (6) -- (2);
\draw[thick] (3) -- (5) -- (4);
\draw[thick] (3) -- (6) -- (4);   
\end{tikzpicture}
\caption{The octahedron and its edge graph.} 
\label{fig:oct}
\end{figure}
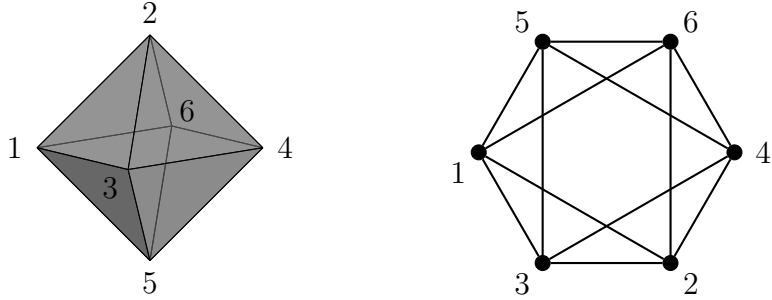

Proposition \ref{prop:localize} now tells us that
the graphic arrangement $\mathcal{A}(G_{\text{oct}})$ is not gentle.
Another computation shows that the ideal quotient
$I = I_0 : (x_1-x_2)$ is a prime ideal, and it hence equals
the likelihood ideal $I = I(\cA(G_{\text{oct}}))$.
The ideal $I$ differs from $I_0$ by only one additional generator $f \in R$ of
degree $
\begin{psmallmatrix}
  3\\
  3
\end{psmallmatrix}
$ with 3092 terms.  Computing $P = I_0 : f$ reveals the second minimal prime of
the pre-likelihood ideal $I_0$, and we obtain the prime decomposition
\[  \qquad I_0 \, = \, I \cap P, \quad {\rm where} \quad
P \,=\, \<\,\sum_{ij\in E}s_{ij}\,,\; x_1 - x_6,\, x_2 - x_6,\, x_3 - x_6,\, x_4 - x_6,\, x_5 -
x_6\>.
\]
The linear forms $x_i-x_6$ in $P$ generate the
irrelevant ideal for the ambient space $\PP^4$ of $\cA(G_{\text{oct}})$.
One can further compute that $\pd (\Omega^1(\mathcal{A}(G_{\text{oct}}))) = 2$, so
this arrangement is not tame either.
\end{example}

Example~\ref{ex:octahedron1} is uniquely minimal among non-gentle arrangements.

\begin{theorem} \label{thm:octahedron}
Consider the graphic arrangements for all graphs $G$ with $n \leq 6$ vertices.
With the exception of the octahedron graph, all of these arrangements are gentle.
\end{theorem}

\begin{proof}
We prove this by exhaustive computation using our tools described in
Section~\ref{sec:software}.
\end{proof}

Except for $\cA(G_{\text{oct}})$, all graphic arrangements of graphs $G$ on fewer than
six vertices satisfy $\pd (\Omega^1(\mathcal{A}(G))) = 1$.  The
octahedron gives rise to more non-gentle graphic arrangements.

\begin{corollary} Any graph that contains the octahedron as an induced
subgraph is not gentle.
\end{corollary}

This is a corollary of Proposition \ref{prop:local}, which holds for
all hyperplane arrangements $\cA$, not just graphic ones.  We let
$L(\cA)$ denote the intersection lattice of the hyperplanes
$H_{i} = \{f_{i} = 0\}$ for $f_{i}\in \cA$.  If $X\in L(\cA)$ then the
\emph{localization} of $\cA$ at $X$ is
$\, \cA_X = \{f_{i} \in \cA : X \subseteq H_{i}\}$.  Any arrangement
of a vertex-induced subgraph is a localization in which $X$ is the
intersection over the $H_{i}$ corresponding to the edges of the
induced subgraph.

\begin{proposition}\label{prop:local}
  The localization of a gentle hyperplane arrangement is gentle.
\end{proposition}

\begin{proof} 
Let $\cA$ be a gentle arrangement and $X \in L(\cA)$.
Suppose that $\cA_X = \{f_1,\dots,f_k\}$ and $\cA\setminus\cA_X = \{f_{k+1},\dots,f_m\}$.
Since the $f_{i}$ are linear, the following ideals are prime:
\[
P\,=\,\langle f_1,\dots,f_k\rangle \subset R \quad
{\rm and} \quad \widetilde{P} \,=\, P + \langle s_1,\dots,s_m\rangle \subset R\left[s_1,\dots,s_m\right].
\]
Since $I_0(\cA)$ is prime and $I_{0}(\cA)\subseteq \widetilde{P}$, the
localization $I_0(\cA)_{\widetilde{P}} \subset R[s]_{\widetilde{P}}$
is prime.  We claim
\begin{equation}\label{eq:localized_ideal}
I_0(\cA)_{\widetilde{P}} \,=\, \langle \theta(\ell_\cA) : \theta \in \Der(\cA)_P\rangle
\,=\,\langle \theta(\ell_\cA) : \theta \in \Der(\cA_X)_P\rangle.
\end{equation}
The first equality is by Theorem~\ref{thm:evaluation} since
localization is exact.
The second follows from $\Der(\cA)_P=\Der(\cA_X)_P$ which holds for
localizations of
arrangements~\cite[Example~4.123]{orlik2013arrangements}.

We now prove that $s_i\in I_0(\cA)_{\widetilde{P}} $ for all
$k+1\le i\le m$.  To this end, fix $s_i$, its corresponding linear
form $f_i$ and hyperplane $H_{i}=\{f_i=0\}$ for $k+1\le i\le m$.  By
Lemma~\ref{lem:modules} we have
$\Der(\cA)=R\theta_E\oplus \Der_0(\cA)$ where $\theta_E$ is
the Euler derivation and $\Der_0(\cA)$ is the submodule of
derivations annihilating all linear forms in~$\cA$.  As
$\Der_0(\cA)\subsetneq \Der_0(\cA\backslash f_{i})$ we can
choose
$\theta_{H_{i}}\in \Der_0(\cA\backslash f_{i})\setminus {
  \Der}_0(\cA)$.  Hence $\theta_{H_{i}}(f_i)=g$ for some nonzero
$g\in R$ and $\theta_{H_{i}}(f_j)=0$ for all $j\neq i$.  The
assumption $f_{i}\notin\cA_X$ yields
$\theta_{H_{i}} \in \Der(\cA_X)$.  Using
\eqref{eq:localized_ideal} we obtain
\[
\theta_{H_{i}}(\ell_\cA)\,=\,s_i\frac{g}{f_i}\in I_0(\cA)_{\widetilde{P}}.
\]
As $I_0(\cA)_{\widetilde{P}}$ contains no polynomials that lie in $R$,
we get $g/f_{i}\notin I_{0}(\cA)_{\widetilde{P}}$.  Thus
$s_i\in I_0(\cA)_{\widetilde{P}}$.  Then the quotient
$I_0(\cA)_{\widetilde{P}}/\langle s_i: k+1\le i\le m\rangle$ is also
prime and by \eqref{eq:localized_ideal} equals
\[
	\bigl\langle \theta(\ell_{\cA_X}):\theta \in \Der(\cA_X)_P\bigr\rangle
	\,\,\subset\,\, R[s_1,\dots,s_k]_{P+\langle s_1,\dots,s_k\rangle}.
\]
The preimage of this ideal in $R[s_1,\dotsc,s_k]$ is the
prime ideal~$I_{0}(\cA_{X})$.
% $I_0(\cA_X)=\langle \theta(\ell_{\cA_X}):\theta \in {\rm
% Der}(\cA_X)\rangle$ is prime.
Hence $\cA_X$ is gentle.
\end{proof}

This argument just made is  independent of $\cA$ being linear.
Hence, for any gentle arrangement of hypersurfaces $\cA$
and a prime ideal $P\subset R$ the subarrangement $\cA\cap P$ is gentle.

Since induced subgraphs give rise to localizations,
Proposition~\ref{prop:local} is one ingredient in the following
conjectural characterization of graphic arrangements that are gentle.

\begin{conjecture}\label{conj:gentle}
A graphic arrangement $\cA(G)$ is gentle if and only if $G_{\text{oct}}$ cannot be obtained from $G$ by a series of edge contractions of
an induced subgraph of $G$.
\end{conjecture}

This conjecture is supported by Theorem~\ref{thm:octahedron}.  A proof
would require not only localizations but also restrictions to a given
hyperplane which in the graphic case correspond to edge contraction.
For general linear arrangements, restrictions do not preserve
gentleness, though.

\begin{proposition}
Restrictions of gentle hyperplane arrangements need not be gentle.
\end{proposition}

\begin{proof}
Edelman and Reiner \cite{ER} constructed a free arrangement of
 $21$ hyperplanes in $\PP^4$ with a restriction to
 $15$ hyperplanes in $\PP^3$ which is not free.
 The linear forms in that nonfree arrangement $\cA$
 are all subsums of $x_1+x_2+x_3+x_4$ which is the $4$-dimensional resonance arrangement~\cite{Resonnace}.
 This $\cA$ is not tame.
 The pre-likelihood ideal
 $I_0(\cA)$ has five minimal generators. The ML degree is $51$.
 Using the \Mtwo tools in Section~\ref{sec:software},
 we find that the ideal quotient $I_0(\cA):x_1$ strictly contains $I_0(\cA)$.
 Therefore, $\cA$ is not gentle.
\end{proof}

Restriction of $\cA(G)$ at a hyperplane models contraction of an edge
in~$G$. This preserves chordality.  Thus restrictions of free graphic
arrangements are free by the characterization.  Therefore, every
restriction of a gentle graphic arrangement could still be gentle.

We now come to the second main result in this section,
a combinatorial construction of generators for
the pre-likelihood ideal $I_0(\cA(G))$ of any graph $G$. Consider the derivations
\[
\theta_k \,=\, x_1^{\, k}  \,\partial_{x_1} + 
x_2^{\, k} \,   \partial_{x_2} + \, \cdots \,+ x_n^{\,k}\,   \partial_{x_n} 
\qquad {\rm for} \quad k = 0,1,\ldots,n-1.
\]
Saito \cite{Saito} proved that  $\{\theta_0,\theta_1,\dotsc, \theta_{n-1} \}$ is a basis
of the free module $\Der(\cA(K_n))$. Before removing edges from $K_n$, it is instructive
to contemplate Theorem \ref{thm:evaluation} for Saito's derivations.

\begin{example}
The log-likelihood function for the braid arrangement $\cA = \cA(K_n)$ equals
\begin{equation}
\label{ex:Knloglike}
 \ell_\cA \quad = \,\sum_{1 \leq i < j \leq n} s_{ij} \cdot {\rm log}(x_i-x_j) .
 \end{equation}
By applying the derivation $\theta_k$ to that function, we obtain a polynomial
in $\CC[x,s]$, namely
\begin{equation}
\label{eq:braididealgens}
\theta_k (\ell_\cA) \,\,\, = \,
\sum_{1 \leq i < j \leq n}  \left(\,\sum_{\ell= 0}^{k-1} x_i^\ell \,x_j^{k-1-\ell} \,\right) \cdot s_{ij} .
\end{equation}
We know from Theorem \ref{thm:evaluation} that these polynomials
generate $I_0(\cA) $, and hence also the likelihood ideal $ I(\cA)$ as
$\cA$ is tame and thus gentle.  For $n=4$ see Examples \ref{ex:braid4}.
\end{example}

Now let $G = (V,E)$ be an arbitrary graph with vertex set $V = [n]$, and let
$\cA = \cA(G)$ be its graphic arrangement. The log-likelihood
function $\ell_\cA$ is  the sum in (\ref{ex:Knloglike}) but now
restricted to pairs $\{i,j\}$ in $E$.
The corresponding restricted sum in (\ref{eq:braididealgens}) still lies
in the ideal $I_0(\cA)$.

A subset $T$ of $[n]$ is a {\em separator} of $G$
if the induced subgraph on $[n] \backslash T$ is disconnected.
We denote this subgraph by $G \backslash T$,
and we consider any connected component $C$ of $G \backslash T$.
Following \cite{LM24}, we define the  \emph{separator-based derivation} associated to the data above:
\[
\theta_{C}^T \,\,\,= \,\,\sum_{i \in C} \prod_{t \in T} (x_i-x_t) \cdot \partial_{x_i}.
\]

The following theorem is implied by the main result in \cite{LM24} along with Theorem \ref{thm:evaluation}.

\begin{theorem} Let $G$ be a graph on $n$ vertices. The module
  $\Der(\cA(G))$ is generated by $\theta_0,\ldots,\theta_{n-1}$
  and a set of separator-based derivations.  Hence, $I_0(\cA)$ is
  generated by the images of $\ell_\cA$ under the derivations
  $\theta_k$ and $\theta_{C}^T$.
\end{theorem}

The generators in this theorem are redundant.
We do not need  $\theta_k$ if $k$ exceeds the connectivity of $G$,
and not all separator-based derivations $\theta_C^T$ are necessary
to generate $\Der(\cA(G))$ and thus $I_0(\cA)$.
It remains an interesting % combinatorial
problem to extract minimal generators.

\begin{example}[Octahedron revisited]\label{ex:octahedron2}
In the case of $G_{\text{oct}}$ from Example~\ref{ex:octahedron1} 
it suffices to consider only (inclusionwise) minimal separators $T$; these are
$\{2,3,5,6\}$, $\{1,3,4,6\}$ and $\{1,2,4,5\}$. The connectivity of the graph is 4.  
The module $\Der(\cA(G_{\text{oct}}))$ is minimally generated by
the following eight derivations:
\[
\theta_0, \theta_1,\,\theta_2,\,\theta_3,\,\theta_4,\,\,
\theta_{\{1\}}^{\{2,3,5,6\}},\,\theta_{\{2\}}^{\{1,3,4,6\}},\,
\theta_{\{3\}}^{\{1,2,4,5\}}.
\]
Setting $z_{ij} \coloneqq x_i-x_j$, we infer the following set of minimal generators
 for the ideal $I_0(\cA)$:
\begin{align*}
\theta_k(\ell_\cA) &\,\, =  \sum_{(i,j) \in E} \left(\,\sum_{\ell= 0}^{k-1} x_i^\ell \,x_j^{k-1-\ell} \,\right) \cdot s_{ij} 
  \quad \mbox{ for }  k=1,\dots,4, \\
  \theta_{\{1\}}^{\{2,3,5,6\}}(\ell_\cA) & \,\,= \,\,z_{13}z_{15}z_{16}\cdot s_{12} + z_{12}z_{15}z_{16}\cdot s_{13}+z_{12}z_{13}z_{16}\cdot s_{15}+z_{12}z_{13}z_{15}\cdot s_{16},\\
  \theta_{\{2\}}^{\{1,3,4,6\}}(\ell_{\cA}) & \,\,= \,\,z_{23} z_{24}z_{26}\cdot s_{12} + z_{21} z_{24}z_{26}\cdot s_{23}+z_{21} z_{23}z_{26}\cdot s_{24}+ z_{21} z_{23}z_{24}\cdot s_{26},\\
  \theta_{\{3\}}^{\{1,2,4,5\}}(\ell_\cA) &\,\, =\,\,z_{32}z_{34}z_{35}\cdot s_{13} + z_{31}z_{34}z_{35}\cdot s_{23}+z_{31}z_{32}z_{35}\cdot s_{34}+ z_{31}z_{32}z_{34}\cdot s_{35}.
\end{align*}
These seven generators are linear in $s$ and they have the $x$-degrees
stated in Example \ref{ex:octahedron1}.  Since
$\theta_0 (\ell_\cA) = 0$, this generator of $\Der(\cA(G_{\text{oct}}))$ does
not yield a generator of~$I_0(\cA)$.
\end{example}

\section{Software and computations}
\label{sec:software}

We have implemented functions
in \Mtwo which compute the
pre-likelihood ideal $I_0(\cA)$
and the likelihood ideal $I(\cA)$
for any arrangement~$\cA$.
The input consists of $m$ homogeneous
polynomials $f_1,\ldots,f_m$ in 
$n$ variables $x_1,\ldots,x_n$.
Along the way, our code creates 
the four polynomial modules
seen in Section \ref{sec:arrangements},
and it also computes the relevant multidegrees.

Our code is made available, along with various examples, in the {\tt
  MathRepo} collection at MPI-MiS via
\url{https://mathrepo.mis.mpg.de/ArrangementsLikelihood}.  In this
section we offer a guide on how to use the software. We present
three short case studies that are aimed at readers from hyperplane arrangements, algebraic statistics, and
particle physics.

We start with the function ${\tt preLikelihoodIdeal}$.  Its input is a
list {\tt F} of $m$ homogeneous elements in a polynomial ring {\tt
  R}. The list {\tt F} defines an arrangement $\mathcal{A}$ in
$\PP^{n-1}$.  Our  code augments the given ring {\tt R} with additional
variables ${\tt s}_1, {\tt s}_2, \ldots, {\tt s}_m$, one for each
element in the list {\tt F}, and it outputs generators for the
pre-likelihood ideal $I_0(\cA)$.  We can then analyze that output and
test whether it is prime, in which case $I_0(\cA) = I(\cA)$.  Our code
also has a function ${\tt likelihoodIdeal}$ which computes $I(\cA)$
directly even if $\cA$ is not gentle.

\begin{example} Revisiting Example  \ref{ex:fourconics}, we consider an arrangement
 $\cA$ of four conics and one line in 
 the projective plane $\PP^2$. We compute its pre-likelihood ideal $I_0(\cA)$ as follows:
 \begin{verbatim}
R = QQ[x,y,z];
F = {x^2+y^2+z^2, x^2+2*y*z-z^2, y^2+2*z*x-x^2, z^2+2*x*y-y^2, x+y+z};
I = preLikelihoodIdeal(F)  \end{verbatim}
The ideal $I_0(\cA)$ has seven minimal generators, starting with
$ 2s_1+2s_2+2s_3+2s_4+s_5$. Our choice of $\cA$ exhibits the generic
behavior in Example \ref{ex:fourconics}. In particular,
the ML degree is~$25$.  Running \verb|codim I, multidegree I, betti mingens I| computes the codimension~$3$, the multidegree $25p^{2}u + 6pu^{2} + u^{3}$ and the total degrees of minimal generators.
% \begin{verbatim}
% codim I, multidegree I, betti mingens I
%             2         2    3         0 1
%      (3, 25T T  + 6T T  + T , total: 1 7)
%             0 1     0 1    1      0: 1 1
%                                   1: . .
%                                   2: . .
%                                   3: . .
%                                   4: . .
%                                   5: . .
%                                   6: . 4
%                                   7: . 2
% \end{verbatim}
A following \verb|isPrime I| returns \verb|true|, which proves
that the arrangement $\cA$ is indeed gentle.

% We next show a more special instance. This has ML degree $9$ and it  is not gentle:
% \begin{verbatim}
% F = {x*y+x*z+y*z, x^2-y*z, y^2-z*x, z^2-x*y, x+y+z};
% I = likelihoodIdeal F
% codim I, multidegree I, betti mingens I, isPrime I
% \end{verbatim}
\end{example}

We now turn to our case studies. The first is 
a non-gentle arrangement of planes in $\PP^3$.

\begin{example}
The following arrangement with $m=9$ is due to Cohen et al.~\cite[Example~5.3]{CDFV}:
\begin{verbatim}
R = QQ[x1,x2,x3,x4];
F = {x1,x2,x3,x1+x4,x2+x4,x3+x4,x1+x2+x4,x1+x3+x4,x2+x3+x4}
ass preLikelihoodIdeal F
I = likelihoodIdeal F;
codim I, multidegree I, betti mingens I, isPrime I
\end{verbatim}
We obtain $I(\cA)$  from $I_0(\cA)$ by
removing the associated prime
$\langle s_1 + s_2 + \cdots + s_9,x_1,x_2,x_3,x_4 \rangle$.
The likelihood ideal $I(\cA)$ has six minimal generators, and
$[\mathcal{L}_\cA] =  5 p^3 u + 9 p^2 u^2 + 5 p u^3 + u^4$.
\end{example}

\begin{example}[No 3-way interaction]\label{no3waygentle}
A model for three binary random variables is given~by
\[
p_{ijk} \,=\, a_{ij}b_{ik}c_{jk} \qquad {\rm for} \,\,  i,j,k \in \{0,1\}.
\]
This parametrizes the toric hypersurface
$\{p_{000}p_{110}p_{101}p_{011} = p_{100}p_{010}p_{001}p_{111}\} \subset \PP^7$.
This toric model fits into our framework by 
setting $m=9$, and considering the $n=12$ parameters
\[
x \,\,= \,\,( a_{00}, a_{10}, a_{01}, a_{11}, b_{00}, \dotsc, b_{11}, c_{00},
\dotsc, c_{11}).
\]
We take $\cA$ to be the $12$ coordinate hyperplanes
$a_{00},a_{10}, \ldots, c_{11} $ together with 
\[
  f(x) \,\,= \,\,
  a_{00} b_{00} c_{00} + 
  a_{00} b_{01} c_{01} + 
  a_{01} b_{00} c_{10} + 
  a_{01} b_{01} c_{11} + 
  a_{10} b_{10} c_{00} + 
  a_{10} b_{11} c_{01} + 
  a_{11} b_{10} c_{10} + 
  a_{11} b_{11} c_{11} .
\]
The pre-likelihood ideal $I_0(\cA)$ has $25$ minimal primes, so the
arrangement is far from gentle. 
The likelihood ideal $I(\cA)$ can be computed for this model as follows:
perform the
saturation $ I_{0}(\cA) : a_{00}f^{2}$ and check that this ideal is
prime. We found this to be the fastest method.

An alternative parametrization of the model with only seven parameters $x_i$
is given by
\[
g(x) \,\, = \,\,
x_1^6+x_1^5x_2+x_1^5x_3+x_1^5x_4+x_1^3x_2x_3x_5
+x_1^3x_3x_4x_6+x_1^3x_2x_4x_7+x_2x_3x_4x_5x_6x_7.
\]
The arrangement $\cA' = \set{x_1,\dotsc, x_7, g(x)}$ is also not
gentle.  The ideal $I_0(\cA')$ has $19$ generators.
The likelihood ideal is $I_{0}(\cA'):x_{1}x_{2}x_{3}x_{4}x_{5}$.  It
has 48 generators in various degrees, some of which are quartic in the
$s$-variables.  The multidegree
$3p^6u+13p^5u^2+25p^4u^3+30p^3u^4+18p^2u^5+6pu^6+u^7$ reveals the
correct ML degree of $3$, known from \cite[Example~32]{ABB}.
\end{example}

\begin{example}[CEGM model]\label{ex:cegm}
Consider the moduli space of six labeled point in linearly general
position in $\PP^2$. This very affine variety arises in the CEGM model in particle physics~\cite{CEGM}.
We write this as the projective
arrangement  $\cA$ with $m=15$ and $n=5$ given
by the $3 \times 3$ minors of the $3 \times 6$~matrix 
\[
  \begin{bmatrix}
    1 & 0 & 0 & 1 & 1 & 1 \\
    0 & 1 & 0 & 1 & x_1 & x_2 \\
    0 & 0 & 1 & 1 & x_3 & x_4 \\
  \end{bmatrix}.
\]
Using $x_5$ for the homogenizing variable, we compute the
pre-likelihood ideal $I_0(\cA)$ as follows:
\begin{verbatim}
R = QQ[x1,x2,x3,x4,x5];
F = {x1,x2,x3,x4,x5,x1-x2,x1-x3,x1-x5,x2-x5,x2-x4,x3-x4,x3-x5,x4-x5,
     x1*x4-x2*x3,x1*x4-x2*x3-x1+x2+x3-x4};
I0 = preLikelihoodIdeal F;
\end{verbatim}
The ideal $I_{0}$ of this arrangement is simple to define, having only
6 generators of degrees $
\begin{psmallmatrix}
  2\\
  1
\end{psmallmatrix}
$ (twice)
and $
\begin{psmallmatrix}
  3\\
  1
\end{psmallmatrix}
$ (four times).  However, due to their size, computing one
Gröbner basis of this ideal is already challenging. Numerically we obtain that $I_0$ has 25 associated primes.
\end{example}

\bigskip
\begin{small}
\noindent {\bf Acknowledgements.}
TK is supported by the Deutsche Forschungsgemeinschaft (DFG, German Research Foundation) within GRK 2297 ``MathCoRe''-- 314838170.
TK and LK are supported by the DFG SPP 2458 ``Combinatorial Synergies'' -- 539866293.
LK and LM are supported by the DFG  TRR 358 -- 491392403.
%LK is supported by the Deutsche Forschungsgemeinschaft (DFG, German Research Foundation) SPP 2458 -- 539866293.
Part of the research was carried out while LK was a member at the Institute for Advanced Study. His stay was funded by the Erik Ellentuck Fellow Fund.

The authors thank Hal Schenck and Julian Vill for helpful discussions.
Last but not least we would like to thank the anonymous referee for several useful suggestions that helped to improve the exposition of the article.
\end{small}

%\bigskip 
\bigskip

\noindent
\footnotesize {\bf Authors' addresses:}

\noindent Thomas Kahle, OvGU Magdeburg, Germany,
{\tt thomas.kahle@ovgu.de}

\noindent Lukas Kühne, Universität Bielefeld, Germany,
{\tt lkuehne@math.uni-bielefeld.de}

\noindent Leonie Mühlherr, Universität Bielefeld, Germany,
{\tt lmuehlherr@math.uni-bielefeld.de}

\noindent Bernd Sturmfels, MPI-MiS Leipzig, {\tt bernd@mis.mpg.de} and
UC~Berkeley, {\tt bernd@berkeley.edu}

\noindent Maximilian Wiesmann, CSBD Dresden, {\tt wiesmann@pks.mpg.de}

\end{document}